\documentclass[12pt]{article}
\usepackage{amsmath}
\usepackage{amsfonts}
\usepackage{amssymb}
\usepackage{ mathrsfs }
\usepackage[pdftex]{graphicx}
\usepackage{amsthm}
\usepackage[hmargin=25mm,vmargin=30mm]{geometry}
\usepackage{setspace}
\usepackage{fancyhdr}
\usepackage{tikz}
\usepackage{tikz-cd}
\usepackage{color}
\usepackage{mathtools}
\usepackage{verbatim}
\usepackage{setspace}
\usepackage{marvosym}
\usepackage{textcase}
\usepackage{enumitem}
\usepackage[colorinlistoftodos, textsize=tiny]{todonotes}
\usepackage{ tipa }

\usepackage{hyperref}
\hypersetup{
    colorlinks,
    citecolor=black,
    filecolor=black,
    linkcolor=black,
    urlcolor=black
}

\title{Evasion Paths by Homotopy Limits}
\author{Gunnar Carlsson, Benjamin Filippenko, Wyatt Mackey}
\date{}
\usepackage[
backend=biber,
style=alphabetic,
sorting=ynt
]{biblatex}

\newtheorem{lem}{Lemma}[section]
\newtheorem{thm}[lem]{Theorem}
\newtheorem{prop}[lem]{Proposition}

\theoremstyle{definition}
\newtheorem{defn}[lem]{Definition}

\newtheorem{rem}[lem]{Remark}
\newtheorem{exmp}[lem]{Example} % Use \begin{exmp} and \end{exmp} in the document--this is here to get compatible numbering with that environment

\fancypagestyle{firststyle}
{
   \fancyhead[L]{}
   \fancyhead[R]{Wyatt Mackey}
}

\tikzset{pullback/.style={minimum size=1.2ex,path picture={
\draw[opacity=1,black,-,#1] (-0.5ex,-0.5ex) -- (0.5ex,-0.5ex) -- (0.5ex,0.5ex);%
}}}
\tikzset{pushout/.style={minimum size=1.2ex,path picture={
\draw[opacity=1,black,-,#1] (-0.5ex,-0.5ex) -- (-0.5ex,0.5ex) -- (0.5ex,0.5ex);%
}}}

\newcommand{\id}{\operatorname{id}}

\newcommand\nc{\newcommand}
\nc{\on}{\operatorname}
\nc\renc{\renewcommand}
\nc{\BR}{\mathbb R}
\nc{\BC}{\mathbb C}
\nc{\BQ}{\mathbb Q}
\nc{\BZ}{\mathbb Z}
\nc{\BN}{\mathbb N}
\nc{\BS}{\mathbb S}
\nc{\BA}{\mathbb A}
\nc{\holim}{\on{holim}}
\nc{\cS}{\mathcal S}
\nc{\cD}{\mathcal D}
\nc{\cE}{\mathcal E}
\nc{\BP}{\mathbb P}
\nc{\Hom}{\on{Hom}}
\nc{\wt}{\widetilde}
\nc{\vspan}{\on{span}}
\nc{\ord}{\on{ord}}
\nc{\im}{\on{im}}
\nc{\Mat}{\on{Mat}}
\nc{\can}{\on{can}}
\nc{\coker}{\on{coker}}
\nc{\ev}{\on{ev}}
\nc{\Tr}{\on{Tr}}
\nc{\End}{\on{End}}
\nc{\swap}{\on{swap}}
\nc{\Set}{\on{Set}}
\nc{\bC}{{\mathbf C}}
\nc{\bc}{{\mathbf c}}
\nc{\bD}{{\mathbf D}}
\nc{\bd}{{\mathbf d}}
\nc{\bE}{{\mathbf E}}
\nc{\be}{{\mathbf e}}
\nc{\bF}{{\mathbf F}}
\nc{\bff}{{\mathbf f}}
\nc{\CE}{\mathcal E}
\nc{\CO}{\mathcal O}
\nc{\CC}{\mathcal C}
\nc{\SC}{\mathscr C}
\nc{\SA}{\mathscr A}
\nc{\SB}{\mathscr B}
\nc{\adj}{\on{adj}}
\nc{\tensor}[3]{#1 \underset{#2}\otimes #3}
\nc{\Nat}{\on{Nat}}
\nc{\op}{\on{op}}
\nc{\Funct}{\on{Funct}}
\nc{\Ob}{\on{Ob}}
\nc{\fR}{\mathfrak{R}}
\nc{\Vect}{\on{Vect}}
\nc{\ns}{\on{non-spec}}
\nc{\ol}{\overline}
\nc{\ul}{\underline}
\nc{\univ}{\on{univ}}
\nc{\Maps}{\on{Maps}}
\nc{\bdd}{\on{bdd}}
\nc{\cont}{\on{cont}}
\nc{\Sym}{\on{Sym}}
\nc{\vol}{\on{vol}}
\nc{\supp}{\on{supp}}
\nc{\Lie}{\on{Lie}}
\nc{\master}{\on{master}}
\nc{\pt}{\on{pt}}
\nc{\bcd}{\[ \begin{tikzcd}}
\nc{\ecd}{\end{tikzcd} \]}
\nc{\funcon}{\on{Funct}_{\on{cont}}}
\nc{\funconpw}{\on{Funct}_{\on{cont},\on{pw-lin}}}
\nc{\ts}{\textsc}
\nc{\codim}{\on{codim}}
\nc{\fm}{\mathfrak m}
\nc{\SR}{\mathscr R}
\nc{\Tor}{\on{Tor}}
\nc{\Ext}{\on{Ext}}
\renc{\sc}[1]{\textsc{#1}}

\begin{document}

\maketitle
\tableofcontents

\section{Introduction}
Suppose a short-staffed collection of guards is patrolling a region---can their lack of manpower be exploited? Or can they devise a patrol pattern that stymies any would-be intruders? This is the general topic of pursuer-evader games. We will view this as a topological problem: at each time $t$, there is a patrolled region $P_t$ and an unpatrolled region $U_t$. By gluing together the copies of $U_t$ (with an appropriate subspace topology) we get a map of spaces $U \to [t_0, t_1]$, delineating how the unpatrolled region changes over time. Navigating through the unpatrolled region, then, is equivalent to finding a section to the map $U \to [t_0, t_1]$. 

More generally, given a map $X \xrightarrow{\pi} B$ of spaces, we would like to be able to compute the homotopy type of the space of sections of $\pi$, $\Gamma \pi = \{s: B \to X \ | \ \pi \circ s = \id_B \}$. The number of distinct sections, as searched for in the patroller-evader games, is $\pi_0(\Gamma\pi)$. These computations are well understood when $\pi$ is a fibration, but the non-fibrant case has received comparatively  little attention. 

\subsection{Prior work}
In the context of mobile sensor networks, this problem was introduced by de Silva and Ghrist in \cite{deSilvaGhristEvasionsFence}. They gave a necessary condition for existence of sections based on the homology of the covered region. Adams and Carlsson proved in \cite{EvasionAdamsCarlsson} that while zig-zag homology can provide a necessary condition, it cannot provide a sufficient condition for the existence of a section. Arone and Jin gave an extension of this in \cite{arone2021applying}, and, by using Goodwillie's calculus of functors, were able to detect the lack of a section in Adams-Carlsson's example. 

Ghrist and Krishnan introduced positive homology in \cite{MR3763757} as a tool which calculates directly the homology of the space of sections. 
When the map from the uncovered region to time is tame, Carlsson and Filippenko in \cite{TameFunctionsPaper} gave a complete computation of the set of connected components of the space of sections in terms of the homotopy groups of the fibers. 

\subsection{Paper outline}
In Section \ref{sec:holim} we give a computation of the homotopy groups of $\Gamma \pi$ for maps of spaces $X \to I$, where $I = [0, 1]$ when $\pi$ satisfies some mild lifting condition; the case $X \to S^1$ is very similar. %An important example of such spaces is when $X \to I$ or $X \to S^1$ is a tame function between smooth manifolds. Section \ref{sec:tame} is devoted to this example, while in Section \ref{sec:sensor_balls} we study the case of sensor ball networks. 
%
%Our computation is an application of the spectral sequence of a homotopy inverse limit (Ch. XI, Section 7 of \cite{bousfieldkan}), which we will very briefly review in Section \ref{sec:holim}. 
The central idea is to subdivide $X \xrightarrow{\pi} I$ into well behaved projections to subintervals of $I$, where it is easier to compute the homotopy groups of the space of sections, and then stitch these together. Let 
\[ 0 = s_0 < s_1 < ... < s_n = 1 \]
be a sequence, so that $\{[s_i, s_{i+1}]\}_{i=0,...,n-1}$ is a partition of the interval. We will write $X_i := \pi^{-1}(s_i)$ for $0 \le i \le n$, and $X_i^{i+1} := \pi^{-1}([s_i, s_{i+1}])$ for $0 \le i \le n-1$. This gives us a diagram of spaces
\[ X_0 \to X_0^1 \leftarrow X_1 \to X_1^2 \leftarrow X_2 \to ... \leftarrow X_n, \]
which we will denote $ZX$. Any section $\gamma: I \to X$ of $\pi$ gives the data of: (1) a point in each $X_i$, and (2) a path in each $X_i^{i+1}$, compatible with the maps in $ZX$ in the obvious way.

On the other hand, the homotopy inverse limit of this diagram is given by 
\begin{enumerate}[label = (\arabic*)]
\item The choice of a point $x_i \in X_i$ for $0 \le i \le n$,
\item the choice of a point $x_i^{i+1} \in X_i^{i+1}$, for $0 \le i \le n-1$,
\item the choice of a path from the image of $x_i$ in $X_i^{i+1}$ to $x_i^{i+1}$, for $0 \le i \le n-1$,
\item and a path from the image of $x_i$ in $X_{i-1}^i$ to $x_{i-1}^i$, for $1 \le i \le n$. 
\end{enumerate}
Since we are interested only up to homotopy, we may combine the data of (3) and (4) as
\begin{enumerate}[label=(3')]
    \item the choice of a path in $X_i^{i+1}$ from the image of $x_i$ to the image of $x_{i+1}$, for $0 \le i \le n-1$. 
\end{enumerate}
There is therefore a map of spaces $\Gamma\pi \to \on{holim} ZX$. In Section \ref{sec:limit_vs_sec}, we develop conditions on the partition and the map $X \to I$ under which this map becomes an $n$-equivalence, for $n \in \BZ$. In such examples, we can then compute the low homotopy groups of $\Gamma\pi$ by studying the much more approachable $\on{holim} ZX$. %The interesting case is whenever it is possible to choose the $s_i$ such that any such path is homotopic, rel. $\{\gamma_i(0), \gamma_i(1)\}$, to a section of $X_i^{i+1} \to [s_i, s_{i+1}]$. For ``nice" projections $X \to I$, we can take a limit over subdivisions of $I$ until these sequential path spaces have the same homotopy type as the space of sections. When we are only interested in the low degree homotopy groups of $\Gamma\pi$, we can impose fewer conditions to require only the low degree homotopy groups of $\Gamma\pi$ and $\on{holim} ZX$ be isomorphic. These conditions are explored in Section 2. 

In particular, the Bousfield-Kan spectral sequence for homotopy inverse limits (Ch. XI, Section 7 of \cite{bousfieldkan}) computes the homotopy groups of the homotopy inverse limit of $ZX$. 

%For instance, we can compute the number of sections, up to homotopy, as
\begin{thm}\label{thm:pi0seq}
Let $\{s_i\}_{i=0,...,n}$ be a partition of the interval, and $X \xrightarrow{\pi} I$ be locally 0-sectional relative to an open cover that contains $[s_i, s_{i+1}]$, and $ZX$ be as above. Then there is a short exact sequence of sets
\[ * \to R^1 \lim \pi_1(ZX) \to \pi_0 \Gamma\pi \to \lim \pi_0(ZX) \to *, \]
where $R^1\lim$ denotes the first right derived functor of the limit functor.
\end{thm}
We prove this in Section \ref{sec:limit_vs_sec}. 

For diagrams of the same form as $ZX$, we can describe $R^1\lim \pi_1(ZX)$ as the orbit set of a group action of $\prod_0^n \pi_1(X_i)$ on $\prod_0^{n-1} \pi_1(X_i^{i+1})$, despite the fact that the $\pi_1(X_i)$ need not be abelian.
\begin{rem}\label{rem:bpt_groupoid}
We are working without basepoints here, so the use of $\pi_1(ZX)$ demands some explanation. Its meaning here is closer to that of the fundamental groupoid than a true fundamental group. The ``valid" extensions associated to each sequential path space $\gamma \in \lim \pi_0(ZX)$ are $R^1\lim \pi_1(ZX, \gamma)$, where we travel along $\gamma$ to define the isomorphism in
\[ \pi_1(X_i, \gamma(i)) \to \pi_1(X_i^{i+1}, \gamma(i)) \cong \pi_1(X_i^{i+1}, \gamma(i+1)) \leftarrow \pi_1(X_{i+1}, \gamma(i+1). \]
Unfortunately, since this level of the spectral sequence deals only with pointed sets, the computations must be done pointwise over the elements of $\lim \pi_0(ZX)$. Of course, if all we care about is whether $\pi_0 \Gamma$ is empty, no computations of $R^1 \lim \pi_1(ZX)$ are necessary. 
\end{rem}

The condition of being ``locally 0-sectional" essentially means that the section space really is the same as the piecewise path space constructed by the homotopy inverse limit--or at least, it is similar enough to correctly compute $\pi_0$. We give a full discussion in Section \ref{sec:limit_vs_sec}. 

Section \ref{sec:tame} is devoted to tame maps of manifolds (Definition \ref{dfn:tamefunction}), similar to Morse functions. We use the smooth structure to show that tame maps are locally 0-sectional, so that their section space can be computed as in Theorem \ref{thm:0_sectional}. This section is drawn from \cite{TameFunctionsPaper}, which gives an explicit computation using the structure of tame maps. This section was particularly due to the second author. 

Section \ref{sec:sensor_balls} is devoted to mobile sensor networks. Sensor ball evasion has been studied several times in the past, after being introduction by de Silva and Ghrist in \cite{deSilvaGhristEvasionsFence}. In this game, sensors wander through a region, observing certain areas and noticing when they overlap with other sensors. We want to know when there is a path through the uncovered region. 

We provide a new calculation of the path components of the uncovered region, using Theorem \ref{thm:pi0seq}. This relies on the computation of unstable homotopy groups of the fiber spaces, however a necessary condition can be extracted from just the homology calculations: there is an easy dictionary between $\pi_0$ and $H_0$, so we can restate $\lim \pi_0(ZX)$ in terms of the cohomolology groups of the uncovered region, which Alexander duality allows us to turn into a question about the homology groups of the covered region. This allows us to provide a new necessary and sufficient condition for the existence of a section, in terms of the fiberwise homology of the covered region.

\section{Holim and spaces of sections}\label{sec:holim}
We suggested above that the homotopy inverse limit of diagrams is relatively computable, and can be used as a good approximation of the space of sections of a map $X \xrightarrow{\pi} I$. This makes it a natural candidate for helping us understand $\Gamma \pi$. To this end, we first give a brief overview of homotopy inverse limits in Section \ref{subsec:intro_to_holim}. 

Section \ref{subsec:path_spaces} details how to go from a map of spaces $X \xrightarrow{\pi} B$ to a diagram whose homotopy inverse limit is a good approximation for $\Gamma\pi$. There is no extra difficulty in working over general spaces $B$, so we work in increased generality there.

Section \ref{subsec:BK_SS} then records the results of the Bousfield-Kan spectral sequence when applied to the diagrams from Section \ref{subsec:path_spaces}. In particular, we show it collapses on the $E_2$ page for spaces $X\xrightarrow{\pi} I$ parameterized over the interval. 

Section \ref{sec:limit_vs_sec} addresses the question: how good an approximation is this homotopy inverse limit to the space of sections? We write down some lifting diagrams which make the natural map $\Gamma\pi \to \on{holim} ZX$ an $n$-equivalence, and then deduce Theorem \ref{thm:0_sectional}. 

\subsection{Homotopy inverse limits}\label{subsec:intro_to_holim}
There are many constructions of the homotopy inverse limit. We will follow \cite{bousfieldkan}, which essentially writes limits as a Hom set in the category of diagrams, and then takes a cofibrant replacement of the source. A somewhat more common construction starts the same, but takes a fibrant replacement of the target rather than a cofibrant replacement of the source. 

More precisely, let $D$ be a small category, $F: D \to \sc{Spaces}$ a diagram indexed by $D$. The functor category $\sc{Funct}_D := \Hom_{\sc{Cat}}(D, \sc{Spaces})$ is enriched over $\sc{Spaces}$, either via the compact open topology when $\sc{Spaces}$ is some topological category, or by setting
\[ \on{Maps}_{\sc{Funct}_D}(F, G)_n := \Hom(\Delta^n \times F, G), \]
when $\sc{Spaces}$ is the category of simplicial sets. Then one definition of the inverse limit of the diagram $F$ is
\[ \lim_{\leftarrow} F := \on{Maps}_{\sc{Funct}_D}(*_D, F),  \]
where $*_D: D \to \sc{Spaces}$ is the functor which sends every object of $D$ to a terminal object of $\sc{Spaces}$. 

The homotopy inverse limit mimics this definition, after taking a cofibrant replacement of $*_D$. For $d$ an object of $D$, we define a space $D/d$, whose $n$-simpleces are sequences
\[ d_n \xrightarrow{\alpha_{n-1}} d_{n-1} \to ... \xrightarrow{\alpha_0} d_0 \xrightarrow{\alpha} d \]
in $D$. Given a map $d \xrightarrow{\phi} d'$ in $D$, there is a map $D/d \to D/d'$ which takes 
\[ (d_n \xrightarrow{\alpha_{n-1}} d_{n-1} \to ... \xrightarrow{\alpha_0} d_0 \xrightarrow{\alpha} d) \mapsto (d_n \xrightarrow{\alpha_{n-1}} d_{n-1} \to ... \xrightarrow{\alpha_0} d_0 \xrightarrow{\phi \circ \alpha} d'). \]
This assembles to a functor
\[ D/-: D \to \sc{Spaces}, \]
which will be our cofibrant replacement. We can now define the homotopy inverse limit,
\begin{defn}
The homotopy inverse limit of $F: D \to \sc{Spaces}$ is the mapping space
\[ \on{holim} F := \on{Maps}_{\on{Funct}_D}(D/-, F). \]
\end{defn}

\subsection{Path spaces}\label{subsec:path_spaces}
Let $X \xrightarrow{\pi} B$ be a map of simplicial sets. We want a nice way to organize the data of $X$ into a functor that we can analyze, similar to $ZX$ in the introduction. To this end, consider the diagram $\wt{B}$, whose objects are the simplices of $B$, and whose maps are inclusions of simplices. Given a simplicial set $S$, let $H(S)$ denote its barycentric subdivision. Then $\wt{B}/-$ assigns to each $n$-simplex of $B$ the barycentric subdivision of $\Delta^n$, $H(\Delta^n)$, and sends the inclusion maps $\Delta^m \hookrightarrow \Delta^n$ to the corresponding inclusions of their barycentric subdivisions. 

Let $\wt{\pi}: \wt{B} \to \sc{Spaces}$ denote the functor which assigns to the each simplex $\sigma$ of $B$ the barycentric subdivision of its preimage $H(\pi^{-1}(\sigma)) \subseteq H(X)$. Then $\on{holim}_{\wt{B}} \wt{\pi}$ is a collection of path data: we have a map $H(\Delta^n)$ to $\pi^{-1}(\sigma)$, and the collection of all of this path data glues together well. Thus $\on{holim} \wt{\pi}$ is a reasonable first approximation to the space of sections of $\pi$. 

When $B$ is not a simplicial set, we can take a simplicial decomposition of $B$, call it $S$, and a similar procedure applies. Let $i: |S| \to B$ be the decomposition relationship between $S$ and $B$, then we get a diagram $\wt{\pi}_S: \wt{S} \to \sc{Spaces}$, by assigning to each simplex $\sigma$ of $S$ the space $\pi^{-1}(i(\sigma))$.

%\section{The space of sections over the interval}
%Let us focus on the case $\pi: X \to I$, where $I$ denotes the unit interval. Let 
%\[ 0 = s_0 < s_1 < ... < s_n = 1 \]
%be a sequence, so that $\{[s_i, s_{i+1}]\}_{i=0,...,n-1}$ is a partition of the interval. Then we will write $X_i := \pi^{-1}(s_i)$ for $0 \le i \le n$, and $X_i^{i+1} := \pi^{-1}([s_i, s_{i+1}])$ for $0 \le i \le n-1$. This gives us a diagram of spaces
%\[ X_0 \to X_0^1 \leftarrow X_1 \to X_1^2 \leftarrow X_2 \to ... \leftarrow X_n, \]
%which we will $ZX$. 
%
%As discussed in Section \ref{subsec:path_spaces}, elements of the homotopy inverse limit of this diagram are given by 
%\begin{enumerate}[label = (\arabic*)]
%\item The choice of a point $x_i \in X_i$ for $0 \le i \le n$,
%\item the choice of a point $x_i^{i+1} \in X_i^{i+1}$, for $0 \le i \le n-1$,
%\item the choice of a path from the image of $x_i$ in $X_i^{i+1}$ to $x_i^{i+1}$, for $0 \le i \le n-1$,
%\item and a path from the image of $x_i$ in $X_{i-1}^i$ to $x_{i-1}^i$, for $1 \le i \le n$. 
%\end{enumerate}
%Since we are interested only up to homotopy, we may combine the data of (3) and (4) as
%\begin{enumerate}[label=(3')]
%    \item the choice of a path in $X_i^{i+1}$ from the image of $x_i$ to the image of $x_{i+1}$, for $0 \le i \le n-1$. 
%\end{enumerate}

\begin{figure}[h]
\centering
\includegraphics{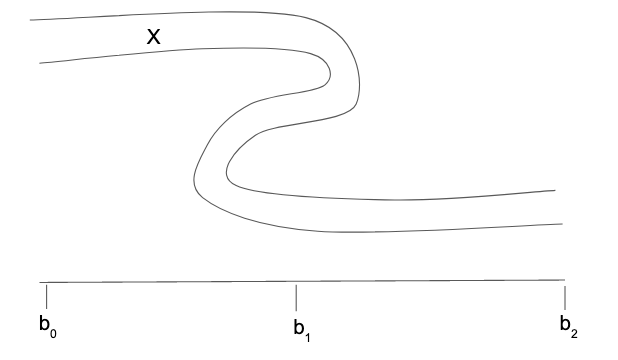}
\caption{$X$ projects down onto $[b_0, b_2]$}
\label{fig:X_proj}
\end{figure}
\begin{exmp}
In Figure \ref{fig:X_proj}, let $X \xrightarrow{\pi} B$ be the projection map. The map $\pi$ clearly has no sections. However, if we used a single 1-simplex $[b_0, b_2]$ to approximate the base the homotopy limit would be the space of paths in $X$ that map the endpoints to the preimage of $b_0$ and $b_2$, which is certainly not empty! Subdividing $B$ further, though, yields better results. 

If we approximate $B$ as the union of two 1-simpleces,  $[b_0, b_1]$ and $[b_1, b_2]$, then $\holim \wt{\pi}$ contains collections of data: a point in the preimage of $b_0$, $b_1$, and $b_2$, and paths in the preimage of $[b_0, b_1]$ and $[b_1, b_2$ that connect these points. But clearly no such points and paths exist, so 
\[ \holim \wt{\pi} \cong \emptyset \cong \Gamma \pi. \]
\end{exmp}

Finer simplicial approximations  in general result in more accurate approximations of the section space. The idea is that, over a smaller space, there are fewer obstructions to a path being homotopic to a section. 

We collect these observations here:
\begin{prop}\label{prop:subdivisions}
Let $X \xrightarrow{\pi} B$ be a map of topological spaces, and let $S$ and $S'$ be simplicial approximations of $B$. 
\begin{enumerate}[label = (\arabic*)]
    \item There is a natural map $\Gamma \pi \to \holim \wt{\pi}_S$.
    \item If $S'$ is a finer approximation of $B$ than $S$ (for instance, if $S' = H(S)$), then the map from (1) factors as
    \[ \Gamma \pi \to \holim \wt{\pi}_{S'} \to \holim \wt{\pi}_S. \]
    \item Let $I$ denote the partially ordered set of simplicial approximations of $B$, ordered by fineness. Then there is a natural map
\[ \Gamma X \to \lim_{S \in I} \holim \wt{\pi}_S. \]
\end{enumerate}
\end{prop}
Points (1) and (2) follow directly from our description of the homotopy inverse limit in this scenario. (3) follows from (1) and (2).%\todo{I removed a bunch of stuff here about parameterized CW complexes, since I couldn't actually make them work. It would be nice to say that the map in (3) is an equivalence in a nice scenario like finite complexes, but that isn't true. Take the identity map $B \to B$, and to the source glue on a loop at some point $b$. Define the projection to take the loop to $b$ as well. This has every reason to count as a complex, and satisfy any other reasonable condition I can impose, but has more sequential paths than it does sections.}

\subsection{Homotopy type of the sequential path space}\label{subsec:BK_SS}
Using the homotopy limit as an approximation to the space of sections would be silly if the approximation didn't have some nice, computable features. This section gives those computation results, mostly referring to the work done in \cite{bousfieldkan}. We will focus on the case where $B$ is an interval, as it simplifies our eventual spectral sequence, and is the primary case of interest anyway. When $B$ is higher dimensional, the spectral sequence still exists---it just no longer collapses on the second page. 

The Bousfield-Kan spectral sequence for homotopy inverse limits computes the homotopy groups of a homotopy inverse limit in terms of the homotopy groups of its constituent parts. The reader interested in the details should look to Chapter XI of \cite{bousfieldkan}. Given a diagram of spaces $F \in \sc{Spaces}^D_*$, this has 
\[ E_2^{s,t} \cong R^s\lim \pi_t(F), \]
where $R^s\lim$ denotes the $s$ right derived functor of inverse limit. 

The differential $d_r: E_r^{s,t} \to E_r^{s+r, t+r-1}$, and the spectral sequence converges so that
\[ \on{gr} \pi_i \on{holim} D \cong \oplus_s E_\infty^{s-i, s}. \]
Applied to $ZX$, note that $R^s\lim ZX$ vanishes for $s \ge 2$ since there are no nondegenerate 2-simplices in the underlying diagram of $ZX$ (XI.6.5 in \cite{bousfieldkan} and the cosimplicial Dold-Kan). Thus our spectral sequence only has two nonzero columns, and collapses on the $E_2$ page. 

\begin{rem}\label{rem:bpt_ZX}
A word should be said on base points: our diagrams $ZX$ cannot have a compatible collection of basepoints on the nose, since the image of the inclusions $X_i, X_{i+1} \to X_i^{i+1}$ do not intersect. Not all is lost, however: fix a section $\sigma: I \to X$. This sets a base points in each of the $X_i$, along with a canonical compatibility between these base points. In this way, we should say that all homotopy groups are relative to a given section: we get a corresponding spectral sequence for each choice. 
\end{rem}

\begin{exmp}\label{exmp:pi0_holim}
If we want to compute $\pi_0 \on{holim} ZX$, it will fit into a ``short exact sequence" of sets
\[ * \to R^1\lim \pi_1(ZX) \to \pi_0 \on{holim} ZX \to  \lim \pi_0(ZX) \to *, \]
by which we mean the fiber of $\pi_0\holim ZX \to \lim \pi_0(ZX)$ over $\gamma \in \lim \pi_0(ZX)$ is isomorphic to $R^1\lim \pi_1(ZX, \gamma)$. 

Then $R^1\lim \pi_1(ZX)$ is the coequalizer of the two maps
\begin{align*}  \prod_{i=0}^n &\pi_1(X_i) \times \prod_{i=0}^{n-1} \pi_1(X_i^{i+1})  \rightrightarrows \\ &\prod_{\alpha: X_i \to X_i^{i+1}} \pi_1(X_{i}^{i+1}) \times \prod_{\beta: X_{i+1} \to X_i^{i+1}} \pi_1(X_i^{i+1}) \times \prod_{i=0}^n \pi_1(X_i) \times \prod_{i=0}^{n-1} \pi_1(X_i^{i+1}). \end{align*}
The first map is induced by $\alpha_*: \pi_1(X_i) \to \pi_1(X_i^{i+1})$, $\beta_*: \pi_1(X_{i+1}) \to \pi_1(X_i^{i+1})$, and identity maps on the remaining factors. The second map is given only by identity maps on the last two factors of the product.%\todo{The quotient calculation just isn't nice? I'm still only half sure I understand it, and have no idea how to write it down so that it's pleasant to read. It simplifies to what's in Ben's paper, which is nicer, but still just a quotient of one thing by the group action of another thing} 
\end{exmp}

\subsection{Comparing the homotopy inverse limit and the space of sections}\label{sec:limit_vs_sec}
Now tooled with a calculation of $\pi_* \holim ZX$, we would like to leverage the map from Proposition \ref{prop:subdivisions},
\[ \Gamma\pi \to \on{holim} ZX, \]
to gain some insight into the homotopy groups of $\Gamma\pi$. In particular: when is this map an $n$-equivalence?

Our general strategy is to investigate the minimal hypotheses we need to get the low dimensional parts of the long exact sequence of a fibration. The homotopy inverse limit of
\[ X_0 \xrightarrow{i} X \xleftarrow{j} X_1 \]
is the space of paths $\gamma: I \to X$, with $\gamma(0) \in i(X_0), \gamma(1) \in j(X_1)$. (In our setting, $i$ and $j$ will always be inclusions, and will often be suppressed from the notation.) We will denote this space by $\on{Hom}((I, 0, 1), (X, X_0, X_1))$. Let $\on{Sec}([a,b], X)$ denote the space of sections of $X \xrightarrow{p} [a,b]$. Then for each choice of homeomorphism $(I, 0, 1) \xrightarrow{f} ([a, b], a, b)$, there is a pullback diagram
\bcd 
\on{Sec}([a,b], X) \ar[r, "f^*"] \ar[d] & \on{Hom}((I, 0, 1) (X, X_0, X_1)) \ar[d, "p_*"] \\
* \ar[r, "* \mapsto f"] & \Hom((I, 0, 1), ([a,b], a, b)). 
\ecd 
\begin{exmp}
If $p$ is a fibration, then so is 
\[p_*: \Hom((I, 0, 1), (X, X_0, X_1)) \to \Hom((I, 0, 1), ([a,b], a, b)).\]
Now, $\Hom((I, 0, 1), ([a, b], a, b))$ is contractible, so if $p_*$ is a fibration, then
\[\on{Sec}([a,b], X) \to \Hom(I, X; \partial I, \partial X)\]
is a weak homotopy equivalence.
\end{exmp}

It is too much to hope in applications of interest that $f^*$ will always be a weak homotopy equivalence. However, in applications we are particularly interested in the first few homotopy groups of $\on{Sec}([a,b], X)$. In this case, we need not require that $p_*$ be a fibration.

\begin{defn}\label{defn:sectional}
%Let $p: X \to [a, b]$ be a continuous map. We will say that $p: X \to U$ is \emph{$n$-sectional} if the map $\on{Sec}([a,b], X) \xrightarrow{i} \Hom(([a, b], a, b), (X, p^{-1}(a), p^{-1}(b)))$ is $n$-connected for each choice of basepoint $\sigma \in \on{Sec}([a, b], X)$. 
Let $p: X \to [a, b]$ be a continuous map. We will say that $p: X \to U$ is \emph{$0$-sectional} if the map $\on{Sec}([a,b], X) \xrightarrow{i} \Hom(([a, b], a, b), (X, p^{-1}(a), p^{-1}(b)))$ is $0$-connected for each choice of basepoint $\sigma \in \on{Sec}([a, b], X)$.
\end{defn}
\begin{defn}\label{defn:locally_sectional}
%We will say $p: X \to I$ is \emph{locally $n$-sectional} if for every $t \in I$, there is a neighborhood $(a, b) \ni t$ such that $p|_{[a,b]}$ is $n$-sectional. If $\mathcal{U} = \{U_i\}_{i \in \alpha}$ is an open cover of $I$, we will say $p$ is locally $n$-sectional subordinate to $\mathcal{U}$ if $p|_{p^{-1}(\ol{U_i})}$ is $n$-sectional for each $U_i \in \mathcal{U}$. 
We will say $p: X \to I$ is \emph{locally $0$-sectional} if for every $t \in I$, there is a neighborhood $(a, b) \ni t$ such that $p|_{[a,b]}$ is $0$-sectional. If $\mathcal{U} = \{U_i\}_{i \in \alpha}$ is an open cover of $I$, we will say $p$ is locally $0$-sectional subordinate to $\mathcal{U}$ if $p|_{p^{-1}(\ol{U_i})}$ is $0$-sectional for each $U_i \in \mathcal{U}$. 
\end{defn}
%The primary application of this definition will be with $U$ a small open set of some space $B$, equipped with $\wt{X} \xrightarrow{\wt{p}} B$ and $X = X|_U, p = \wt{p}|_{X}$. 

With the terminology all defined, we're now ready to prove Theorem \ref{thm:pi0seq}. 
\begin{proof}[{\bf Proof of Theorem \ref{thm:pi0seq}}]
By Proposition \ref{prop:subdivisions} (1), there is a natural map from the space of sections to the homotopy inverse limit, $\Gamma \pi \to \on{holim} \wt{\pi}_{\cup [s_i, s_{i+1}]}$. The Bousfield-Kan spectral sequence gives a computation of $\pi_0\on{holim} \wt{\pi}_{\cup [s_i, s_{i+1}]}$ which matches that claimed in the theorem, so it suffices to show that
\[ \pi_0 \Gamma \pi \to \pi_0 \on{holim} \wt{\pi}_{\cup [s_i, s_{i+1}]} \]
is an isomorphism. 

The condition that $X \xrightarrow{\pi} I$ is locally 0-sectional subordinate to an open cover containing $[s_i, s_{i+1}]$ means that each of the maps
\[ \Gamma \pi|_{[s_i, s_{i+1}]} \to \Hom(([s_i, s_{i+1}], s_i, s_{i+1}), (\pi^{-1}([s_i, s_{i+1}]), \pi^{-1}(s_i), \pi^{-1}(s_{i+1})) \]
induce isomorphisms on $\pi_0$. 

Now elements of $\Gamma \pi$ are the same as elements of the limit
\[ \lim_{\leftarrow} (... \to \pi^{-1}(s_{i-1}) \leftarrow \Gamma \pi|_{[s_{i-1}, s_i]} \to \pi^{-1}(s_i) \leftarrow \Gamma\pi|_{[s_i, s_{i+1}]} \to \pi^{-1}(s_{i+1}) \leftarrow ...) \]
Likewise, elements of the homotopy inverse limit over ZX are the same as elements of the (normal) limit over the diagram
\[ ... \to \pi^{-1}(s_{i-1}) \leftarrow \Hom(([s_{i-1}, s_i], s_{i-1}, s_i), (\pi^{-1}([s_{i-1}, s_{i}]), \pi^{-1}(s_{i-1}), \pi^{-1}(s_{i})) \to \pi^{-1}(s_i) \leftarrow ... \]
Let $A_i := \Hom(([s_{i-1}, s_i], s_{i-1}, s_i), (\pi^{-1}([s_{i-1}, s_{i}])$, for page space. The induced map of diagrams
\bcd 
... \ar[r] & \pi^{-1}(s_{i-1}) \ar[d, "\cong"] & \Gamma \pi|_{[s_{i-1}, s_i]} \ar[l] \ar[d] \ar[r] & \pi^{-1}(s_i) \ar[d, "\cong"] & ... \ar[l] \\
... \ar[r] & \pi^{-1}(s_{i-1}) & A_i \ar[l] \ar[r] & \pi^{-1}(s_i) & ... \ar[l] 
\ecd 
The middle map also induces an isomorphism on connected components by hypothesis, and this property is preserved under taking limits. 
\end{proof}

Let's abstract, and suppose we have
\bcd 
F \ar[r, "i"] \ar[d] & E \ar[d, "p"] \\
* \ar[r] & B,
\ecd 
Note: we are \emph{not} assuming $p$ is a fibration.

\subsubsection{0-sectional maps}\label{subsec:pi0} There is a sequence
\begin{equation}\label{pi_0_case} \pi_0 F \xrightarrow{\pi_0(i)} \pi_0 E \xrightarrow{\pi_0(p)} \pi_0 B. \end{equation}
The map $\pi_0(p)$ is certainly an epi if $E \to B$ is surjective, which is true in our case whenever $X$ contains a path connected component whose map to $I$ is surjective. The composite of (\ref{pi_0_case}) is automatically 0. To be exact, everything sent to the basepoint of $\pi_0 B$ must be homotopic to an element of the image of $F \to E$. 

In our setting, that means that each map $I \to X$ is homotopic to a section. (Technically we should specify that the pullback under $f$ is homotopic to a section, but the choice of $f$ doesn't matter.) We can state this condition by requiring that in the following diagram, there is a lift
\bcd 
I \ar[r] \ar[d, "\on{id} \times 0"] & X \ar[d, "p"] \\
I \times I \ar[ur, dashed, "H"] \ar[r, "h"] & {[a, b]}
\ecd 
where $H(0, t), H(1, t)$ are constant and $h$ is any homotopy. This gives us surjectivity of the map from sections to paths. The diagram we must require for injectivity is similar; we summarise the results in the following theorem: 
\begin{thm}\label{thm:0_sectional}
A map $p: X \to [a,b]$ is 0-sectional if
\begin{enumerate} 
\item The map $\pi_0\on{Sec}([a, b], X) \to \pi_0\Hom(([a,b], a, b), (X, p^{-1}(a), p^{-1}(b)))$ is surjective, i.e. there is a lift
\bcd 
I \ar[r] \ar[d, "\on{id} \times 0"] & X \ar[d, "p"] \\
I \times I \ar[ur, dashed, "H"] \ar[r, "h"] & {[a, b]}
\ecd 
where $H(0, t), H(1, t)$ are constant whenever $h(1, s) = f^*$; and
\item The map $\pi_0\on{Sec}([a, b], X) \to \pi_0\Hom(([a,b], a, b), (X, p^{-1}(a), p^{-1}(b)))$ is injective, i.e. there is a lift
\bcd 
I \times I \ar[r] \ar[d] & X \ar[d, "p"] \\
I \times I \times I \ar[ur, dashed, "H"] \ar[r, "h"] & {[a, b]},
\ecd 
when $h|_{(s, t, 1)}: I \times I \to [a, b]$ is equal to $f \circ \pi_1$, where $\pi_1: I \times I \to I$ is projection onto the first factor. 
\end{enumerate}
\end{thm}
We have already justified the first diagram; the second diagram just says that sections which are homotopic are also fiberwise homotopic. 

We could hope to get a nicer theorem statement by reproducing the proof of the long exact sequence of a fibration, and seeing what diagrams are required for the low dimensional parts of the sequence to hold. Suppose we wanted exactness of $\pi_1 B \to \pi_0 F \to \pi_0 E$, for instance. Here exactness means that $\pi_1B$ acts on $\pi_0 F$, and the quotient by this action includes into $\pi_0E$. 

For $\pi_1(B)$ to have an action on $\pi_0 F$, we need to be able to form a lift as indicated in the diagram
\bcd 
* \ar[r] \ar[d] & E \ar[d] \\
I \ar[r, "\sigma"] \ar[ur, dashed, "\wt{\sigma}"] & B. 
\ecd 
The action is well defined if homotopies of paths also lift: that is, if the dashed arrow exists as in the following diagram for any square
\bcd
I \ar[r, "\wt{\sigma}"] & E \ar[d] \\
I \times I \ar[u, leftarrow, "\on{id} \times 0"] \ar[ur, dashed] \ar[r, "H"] & B.
\ecd 

In our setting, we can use the product/hom adjunction to get an equivalence between the diagram
\bcd 
* \ar[r] \ar[d] & \Hom((I, 0, 1), (X, p^{-1}(a), p^{-1}(b))) \ar[d] \\
I  \ar[r] \ar[ur, dashed] & \Hom((I, 0, 1), ({[a,b]}, a, b))
\ecd 
and the diagram
\bcd 
I \ar[r] \ar[d] & X \ar[d] \\
I \times I \ar[r] \ar[ur, dashed] & {[a,b]},
\ecd 
with the understanding that maps preserve $\partial I$ and $\partial X$. Likewise, the action being well defined  translates to existence of a dashed arrow in the diagram
\bcd 
I \times I \ar[r] \ar[d] & X \ar[d] \\
I \times I \times I \ar[r] \ar[ur, dashed] & {[a,b]}.
\ecd 
These are very nice requirements, which unfortunately often will not hold in the cases in which we are interested: for instance, lifts (for either diagram) need not exist in the case of a manifold mapping to the interval by a tame function with a unique critical point. (See Definition \ref{dfn:tamefunction}.)

\section{Application: Tame functions}\label{sec:tame}
We now give some applications, to show that locally $n$-sectional maps really do show up in nature. This section is devoted to tame functions, a kind of smooth function similar to Morse functions. We construct a ``nice" vector field on a tame function in Section \ref{subsec:tame_local}, then use this to prove that tame functions are locally 0-sectional in Section \ref{subsec:tame_sectional}. We can then conclude from Theorem \ref{thm:0_sectional} and Example \ref{exmp:pi0_holim} 
\begin{thm}\label{thm:tame_main}
Given a tame map $\pi: X \to [s_-, s_+]$, let $ZX$ be the diagram associated to a partition of $[s_-, s_+]$ for which there is at most one critical point of $\pi$ in each $[s_a, s_{a+1}]$. Then there is a short exact sequence of sets
\[ * \to R^1\lim \pi_1(ZX) \to \pi_0 \Gamma \pi \to \lim \pi_0(ZX) \to *. \]
Moreover, $R^1\lim \pi_1(ZX)$ is isomorphic to the orbit set of a group action of $\prod_0^n \pi_1(X_i)$ on $\prod_0^{n-1} \pi_1(X_i^{i+1})$. Basepoints are here suppressed, and we mean them as in Remarks \ref{rem:bpt_groupoid} and \ref{rem:bpt_ZX}.
\end{thm}

The following notion of a tame function, as well as the constructions we perform with them in this section, is inspired by the Morse theory on manifolds with boundary.%; see \cite{MR3203357}\cite{MR3493413}\cite{MR348790}\cite{MR611759}\cite{MR0331409}\cite{MR2388043}\cite{MR2819662}\cite{MR0190942} and Remark~\ref{rmk:tameversusmorse}.

\begin{defn} \label{dfn:tamefunction} Let $X$ be a compact cobordism of manifolds with boundary, and $\pi : X \rightarrow [s_-,s_+]$ a smooth function. Let $B = \partial X \setminus \pi^{-1}(s_-) \cup \pi^{-1}(s_+)$. We say $\pi$ is a \emph{tame function} if it satisfies the following conditions:
\begin{itemize}
\item The critical points of the restriction $\pi|_B : B \rightarrow [s_-,s_+]$ are isolated and have distinct critical values in $(s_-,s_+)$,\\
\item $\pi$ is submersive,\\
\item $X_- = \pi^{-1}(s_-)$ and $X_+ = \pi^{-1}(s_+)$.
\end{itemize}
\end{defn}

\begin{rem} \label{rmk:tameversusmorse}
Definition~\ref{dfn:tamefunction} does not require any nondegeneracy condition on the critical points of $\pi|_B : B \rightarrow [s_-,s_+]$, so $\pi|_B$ does not have to be a Morse function in the sense of \cite[Def.~2.3]{MR0190942}. In this way, our definition is more general than the Morse condition. On the other hand, we do not allow $\pi$ itself to have critical points. Note that for the projection $\pi : X \rightarrow [s_-,s_+]$ from a codimension-$0$ submanifold $X \subset \mathbb{R}^d \times [s_-,s_+]$, as is the situation in the smoothed evasion path problem, the map $\pi$ is submersive.
\end{rem}

There are two types of critical points of a tame function on the boundary.

\begin{defn} \label{dfn:criticalpointtypes}
Let $\pi : X \rightarrow [s_-,s_+]$ be a tame function, $p \in B$ a critical point of $\pi|_B : B \rightarrow [s_-,s_+]$, and $\eta \in T_pX$ an outward pointing vector. Then $p$ is \emph{type N} if $d\pi(\eta) < 0$ and \emph{type D} if $d\pi(\eta) > 0.$
\end{defn}

Note that $d\pi(\eta) = 0$ is impossible since it would imply that $p$ is a critical point of $\pi$.

\subsection{Local constructions}\label{subsec:tame_local}
This section is devoted to creating a nice smooth vector field on tame manifolds. The idea is that, by flowing along that vector field, we can find local homotopies of paths into sections. 

Our constructions are local in the sense that for a tame function $\pi : X \rightarrow [s_-,s_+]$ and $t \in [s_-,s_+$], the constructions apply in the preimage $\pi^{-1}([t-\epsilon,t+\epsilon])$ for $\epsilon > 0$ small enough.

\begin{prop} \label{prp:pseudogradient}
Let $\pi : X \rightarrow [s_-,s_+]$ be a tame function. If $\pi|_B$ has only type D critical points, then there exists a smooth vector field $\xi$ on $X$ with the following properties:
\begin{enumerate}
\item $d\pi(\xi) = -1$ on all of $X$.\\
\item For all\footnote{For $p \in X_-$, the vector field $\xi$ constructed in the proof of Proposition~\ref{prp:pseudogradient} is outward pointing on $Int(X_-)$, and on $\partial X_-$ it is outward pointing with respect to $X_-$ and inward pointing with respect to $B$.} $p \in \partial X \setminus X_-$, the vector $\xi_p \in T_pX$ is inward pointing.\\
\end{enumerate}

If $\pi|_B$ has only type N critical points, then there exists a smooth vector field $\xi$ on $X$ with the following properties:
\begin{enumerate}
\item $d\pi(\xi) = 1$ on all of $X$.\\
\item For all $p \in \partial X \setminus X_+$, the vector $\xi_p \in T_pX$ is inward pointing.\\
\end{enumerate}

Moreover, given any smooth section $\mathfrak{b} : [s_-,s_+] \rightarrow X$ of $\pi$ (i.e.\, $\pi \circ \mathfrak{b}(t) = t$) such that $\mathfrak{b}(t) \not \in B$ for all $t \in [s_-,s_+]$, the vector field $\xi$ can be chosen such that
$$\frac{d}{dt} \mathfrak{b} (t) =
\begin{cases} 
      -\xi|_{\mathfrak{b}(t)}  & \text{if } \text{type D} \\
     \xi|_{\mathfrak{b}(t)}  & \text{if } \text{type N}.
   \end{cases}
$$
\end{prop}
\begin{proof}
We consider the case of type D critical points; the type N case is symmetric.

It suffices to construct $\xi$ locally in an open neighborhood of every point $p \in X$ and then sum up these local vector fields with a partition of unity. For the local construction, let $p \in X$ and for now assume that if $p \in B$ then $p$ is a regular point of $\pi|_B$. Then by the implicit function theorem there exists an open neighborhood $U(p) \subset X$ of $p$ and a smooth chart $\varphi : U(p) \xrightarrow{\sim} V$ such that $\varphi(p) = (\pi(p),0,\ldots,0) \in V$ where $V$ is an open subset of one of the following spaces depending on where $p$ sits on $X$, and in such a way that
$$\pi \circ \varphi^{-1} : V \rightarrow [s_-,s_+]$$
is the projection onto the first coordinate:
\begin{itemize}
\item If $p \in Int(X)$, then $V \subset \mathbb{R}^{\dim X}$,\\
\item If $p \in Int(X_-)$, then $V \subset [s_-,\infty) \times \mathbb{R}^{\dim X -1}$,\\
\item If $p \in Int(X_+)$, then $V \subset (-\infty,s_+] \times \mathbb{R}^{\dim X -1}$,\\
\item If $p \in Int(B)$, then $V \subset \mathbb{R} \times ([0,\infty) \times \mathbb{R}^{\dim X -2})$,\\
\item If $p \in \partial X_-$, then $V \subset [s_-,\infty) \times ([0,\infty) \times \mathbb{R}^{\dim X -2})$,\\
\item If $p \in \partial X_+$, then $V \subset (-\infty,s_+] \times ([0,\infty) \times \mathbb{R}^{\dim X -2})$.\\
\end{itemize}
In all cases above, the constant vector field
$$(-1,1,0,\ldots,0)$$
on $V$ pulls back through the diffeomorphism $\varphi$ to a vector field $\xi$ on $U(p)$ satisfying $d\pi(\xi) = -1$, and moreover it is inward pointing along all points $p \in \partial X \setminus X_- = Int(B) \cup X_+$, as required.

It remains to consider a critical point $p \in B$ of type $D$. By definition of tame function, $p$ is in $Int(B) = B \setminus (\partial X_- \cup \partial X_+)$. Consider a neighborhood $\tilde{U}(p) \subset X$ of $p$ and a coordinate chart $\tilde{U}(p) \xrightarrow{\sim} V \subset [0,\infty) \times \mathbb{R}^{\dim X -1}$ that sends $p$ to $0$. Then the constant vector field $(1,0,\ldots,0)$ pulls back to a vector field $\xi$ on $\tilde{U}(p)$ that is inward pointing, and hence $d_p\pi(\xi_p) < 0$ since $p$ is type D. Then $d\pi(\xi) < 0$ in a smaller open neighborhood $U(p) \subset \tilde{U}(p)$. Hence the vector field $\xi / |d\pi(\xi)|$ satisfies (i) and (ii) on $U(p)$, as required.

Consider now the final statement of the proposition where we are given a smooth section $\mathfrak{b}$ that is disjoint from $B$. For $p \neq \mathfrak{b}(t)$ for all $t$, choose the neighborhood $U(p)$ to be disjoint from the image of $\mathfrak{b}$ and perform the construction as above. Suppose $p = \mathfrak{b}(t)$ for some $t$. Choose $U(p)$ to be disjoint from $B$. Define $\xi|_{\mathfrak{b}(t)} = -\mathfrak{b}'(t)$ for all $t$ such that $\mathfrak{b}(t) \in U(p)$. Then $d\pi(\xi|_{\mathfrak{b}(t)}) = -(\pi \circ \mathfrak{b})'(t) = -1$. Extend $\xi$ over $U(p)$ so that $d\pi(\xi) = -1$ on $U(p)$.
\end{proof}

\subsection{Tame functions are locally 0-sectional}\label{subsec:tame_sectional}
We now will use the vector field constructed in Proposition \ref{prp:pseudogradient} to show that tame functions are locally 0-sectional. The idea is to work in subdivisions that contain at most one critical point, and flow our path along the pseudogradient. This produces a homotopic path which ``fails to be a section" entirely on one side of a critical point. But there, $f$ is a fibrations, and so the path can be straightened. 

It will suffice to show that a subinterval $[a, b] \subset [s_-, s_+]$ with $p|_{p^{-1}((a,b))}$ having exactly 1 critical point is 0-sectional. We consider only the type D case, as the type N case is symmetric. Then if we can show that the conditions of Theorem \ref{thm:0_sectional} are satisfied, then we can successfully compute $\pi_0 \Gamma X$ by the formula of Theorem \ref{thm:pi0seq}. 

\begin{prop}\label{prop:tame_are_sectional}
Suppose $X \xrightarrow{\pi} [a, b]$ is a tame function with a unique critical point of type D. Then $\pi$ is 0-sectional, i.e. every path $([a, b], a, b) \xrightarrow{\gamma} (X, \pi^{-1}(a), \pi^{-1}(b))$ is homotopic to a section, and this section is unique up to homotopy of sections. 
\end{prop}
\begin{proof}
First, choose a homeomorphism $([a, b], a, b) \xrightarrow{f} ([a, b], a, b)$ such that $\pi \circ \gamma \circ f(c) \ge c$ for every $c \in [a, b]$. Then define $w(c)$ to be the point in $X$ that comes from flowing $\gamma \circ f(c)$ along $(\pi \circ \gamma \circ f(c) - c)\xi$ for 1 unit of time, where $\xi$ is the vector field from Proposition \ref{prp:pseudogradient}. Then $w: [a, b] \to X$ is a section, and by construction it is homotopic to $\gamma$. Note that if $\gamma$ is already a section, we can choose $f$ to be the identity map and get $w = \gamma$. %Call this procedure $\Psi: \Hom(([a, b], a, b), (X, \pi^{-1}(a), \pi^{-1}(b))) \to \Gamma\pi$, $\Psi(\gamma) :=w$. 

Now suppose $\gamma$ is also homotopic to a section $w': [a, b] \to X$. We need to show that $w$ and $w'$ are homotopic as sections: that is, there is a homotopy $H: [a, b] \times I \to X$ such that $H_t$ is a section for each $t \in I$. Let $G: [a, b] \times I \to X$ be a homotopy from $\gamma$ to $w'$. Then we can choose a continuous collection of homeomorphisms $f_t: ([a, b], a, b) \to ([a, b], a, b)$ such that $\pi \circ G_t \circ f_t(c) \ge c$ for $c \in [a, b]$, and $f_1$ is the identity map, and $f_0$ the same $f$ we used in constructing $w$. Applying the procedure of the preceding paragraph, we obtain a homotopy of sections $w_t$ from $w_0 = w$ to $w_1 = w'$. 
\end{proof}

We can now conclude Theorem \ref{thm:tame_main} as a corollary of Proposition \ref{prop:tame_are_sectional} and Theorem \ref{thm:pi0seq}.  

\section{Application: sensor ball evasion}\label{sec:sensor_balls}
We return to the setting of the introduction now: if an understaffed collection of guards is patrolling a space, when can we avoid them? Sensor ball evasion is a model of this problem, in which our guards have a reasonably controlled perception--for instance, they are motion detectors that observe some ball centered on themselves. 

Given a collection of continuous sensors $\cS = \{ \gamma_j : [0,1] \rightarrow \cD \}_{j \in J}$ moving in a bounded domain $\cD \subset \mathbb{R}^d$, an \emph{evasion path} is a continuous map $\delta : [0,1] \rightarrow \cD$ that avoids detection by the sensors for the whole time interval $I = [0,1]$. Suppose each sensor $\gamma_j$ observes a ball of fixed radius in $\mathbb{R}^d$ centered at $\gamma(t)$ at every $t \in I$, and let $C_t$ be the time-varying union of these sensor balls. Then the intruder $\delta$ avoids detection if $\delta(t)$ is not in $C_t$ for all $t \in I$.

The sensor ball evasion path problem asks for a criterion that determines whether or not an evasion path exists. Ideally, we would do this with as little information as possible: a common assumption is that the sensors only detect other nearby sensors and intruders; in particular, they do not know their coordinates in $\mathbb{R}^d$. The existence problem simply asks if the space of evasion paths $\cE$ is empty or not. We might also ask after the homotopy type of $\cE$, if it isn't empty. 

This problem was first stated and studied in \cite{deSilvaGhristEvasionsFence} by de Silva and Ghrist in dimension $d = 2$. Provided that the time-varying \v{C}ech complex changes only at finitely many times, and by only simple changes at those times, the authors establish a necessary condition for existence of an evasion path which states that the connecting homomorphism on a relative homology group with respect to the fence of the covered region must vanish.

In \cite{EvasionAdamsCarlsson}, Adams and Carlsson consider the same evasion problem in arbitrary dimension $d \geq 2$, and they establish a necessary condition for the existence of an evasion path based on zigzag persistent homology \cite{MR2657946} of the time-varying \v{C}ech complex. They also explain how this result is equivalent to a generalization of de Silva and Ghrist's result to the case $d \geq 2$.

In both de Silva-Ghrist and Adams-Carlsson, the necessary condition for existence of an evasion path is determined and easily computable from the overlap information of sensor balls. Essentially, sensors only need to know which other sensors are nearby; they don't need to know their coordinates in $\mathbb{R}^d$.

Adams-Carlsson also study smarter sensors in the plane $\mathbb{R}^2$ that know local distances to nearby sensors and the natural counterclockwise ordering on nearby sensors. For these smarter sensors in $\mathbb{R}^2$, they derive a necessary and sufficient condition for existence of an evasion path as well as an algorithm to compute it, under the assumption that the covered region $C_t$ is connected for all $t \in I$. They show that this connectedness assumption is necessary for their methods.

In \cite{MR3763757}, Ghrist and Krishnan define positive homology and cohomology for directed spaces over $\mathbb{R}^q$, and they compute it using sheaf-theoretic techniques. They consider a general type of evasion problem where the covered region is a pro-object in a category of smooth compact cobordisms. They derive a criterion on positive cohomology that is necessary and sufficient for existence of an evasion path, under the assumption that $C_t$ is connected. 

In Section \ref{subsec:priorwork}, we establish some notation and formally state the problem. Section \ref{subsec:sensors_0_sec} then shows that the unpatrolled region is locally 0-sectional, so that we can solve the existence problem by Theorem \ref{thm:pi0seq}. However, this gives a solution in terms of the unpatrolled space, which we don't always assume we have access to: we would like to use less information, such as just the \u{C}ech complex associated to the sensor balls. We can accomplish thisby using Alexander Duality, which is the purpose of Section \ref{subsec:sensors_homology}. Our results are then summarized in Theorem \ref{thm:sensor_ball_sections}. 

\subsection{Notation} \label{subsec:priorwork}
Fix some sensing radii $r_j > 0$ for each $j \in J$, and say that each sensor $\gamma_j \in \cS$ can detect objects within the closed ball
$$B_{\gamma_j(t)} = \{ x \in \mathbb{R}^d \,\, | \,\, |\gamma(t) - x| \leq r_j \}$$
at all times $t \in I = [0,1]$. The time-$t$ covered region
\begin{equation} \label{eq:sensorballunion}
C_t = \bigcup_{\gamma \in \cS} B_{\gamma(t)} \subset \mathbb{R}^d
\end{equation}
is the union of the sensor balls, and it is homotopy equivalent to the \v{C}ech complex of the covering of $C_t$ by the sensor balls. The time-$t$ uncovered region is the complement
$$X_t = \cD \setminus C_t,$$
where $\cD \subset \mathbb{R}^d$ is a bounded domain that is homeomorphic to a ball.

Assume that the boundary $\partial \cD$ is covered by a collection of immobile fence sensors $F\cS \subset \cS$, meaning $\gamma(t) \in \mathbb{R}^d$ is constant for $\gamma \in F\cS$, and that the union of balls $B_{\gamma(t)}$ for $\gamma \in F\cS$ covers the boundary $\partial \cD$ and is homotopy equivalent to $\partial \cD$. In particular, this ensures that an intruder $\delta$ can never escape from the domain $\cD$.

We define the covered region
\begin{equation} \label{eq:coveredregion}
C := \bigcup_{t \in I} C_t \times \{t\} \subset \mathbb{R}^d \times I
\end{equation}
and the uncovered region
$$X := (\cD \times I) \setminus C = \bigcup_{t \in I} X_t \times \{t\}.$$
Then an evasion path is equivalent to a continuous section $\delta : I \rightarrow X$ of the projection $\rho_{X} : X \rightarrow I$, i.e.,\ $\rho_{X} \circ \delta = id_I$.

The sensor ball evasion path problem asks for a criterion for existence of an evasion path.

\subsection{Sensor ball covered and uncovered regions are locally 0-sectional}
\label{subsec:sensors_0_sec}

Both de Silva-Ghrist and Adams-Carlsson simplify the problem by looking at the \u{C}ech complex associated to the union of sensors $C_t$, and the nerve theorem tells us this simplicial complex computes the same homology. Moreover, it is assumed that $C_t$ varies at only finitely many times $\{ t_1, ..., t_n\}$, and moreover, at any time $t_i, 0 < i < n$, $C_t$ changes by either the addition of simplices or the removal of simplices, but not both. In this section, we will show that under these hypotheses, both $X$ and $C$ are locally 0-sectional, subordinate to any open cover whose open sets contain at most one of the $t_i$. 

Let there be $\{s_i\}_{i=0,...,n}$ with $0 = s_0 < t_1 < s_1 < ... < t_n < s_n = 1$. Up to homotopy equivalence, we may replace $C$ by what Adams-Carlsson refer to as the \emph{stacked \u{C}ech complex}: $SC \to I$ is a space equipped with a projection to the interval, with
\[ SC := \coprod_{i=0,...,n} C_{s_i} \times [t_i, t_{i+1}] / \sim, \]
where $t_0 = 0, t_{n+1} = 1$, and $\sim$ identifies either $C_{s_i} \times \{t_{i+1}\}$ with a subset of $C_{s_{i+1}} \times \{t_{i+1}\}$, or $C_{s_{i+1}} \times \{t_{i+1}\}$ with a subset of $C_{s_i} \times \{t_{i+1}\}$ for each $i$. The map to $I$ is given by projection onto the second factor of the products. 

%The definition is nice, but the complement isn't quite a stacked cech complex--or rather, it probably is, but the proof is longer than the part of the proof we care about. 
\begin{defn}
We will say that $Y \to I$ is a \emph{stacked simplicial complex} if there are distinct points $t_1, ..., t_n \in I$ so that 
\[ Y \cong \coprod_{i=0,...,n} Y_{s_i} \times [t_i, t_{i+1}] / \sim, \]
where $s_i \in (t_i, t_{i+1})$, $Y_{s_i}$ are simplicial complexes, and $\sim$ identifies either $Y_{s_i} \times \{t_{i+1}\}$ with a subset of $Y_{s_{i+1}} \times \{t_{i+1}\}$, or $Y_{s_{i+1}} \times \{t_{i+1}\}$ with a subset of $Y_{s_i} \times \{t_{i+1}\}$ for each $i$.

We will sometimes say $Y$ is a \emph{stacked \u{C}ech complex} when $Y$ is a stacked simplicial complex when we want to emphasize that the simplicial complex structure $Y_{s_i}$ comes from a \u{C}ech nerve. 
\end{defn}
%
%\begin{lem}
%Let the letters be as above: we are in a domain $\mathcal{D}$, $C \subseteq D$ is the covered region associated to moving sensor balls, $X = \mathcal{D} \setminus C$ is the uncovered region, and $SC$ is the stacked \u{C}ech complex associated to $C$. Then $X$ is homotopy equivalent to a stacked \u{C}ech complex 
%\end{lem}

\begin{prop}\label{prop:stacked_are_sectional}
Let $Y \xrightarrow{p} I$ be a stacked \u{C}ech complex, with $t_i$ and $s_i$ as above. Then $p$ is locally 0-sectional, subordinate to any open cover whose open sets contain at most one of the $t_i$.
\end{prop}
\begin{proof}
%The argument is very similar to that for tame functions. Suppose $[a, b]$ is a subinterval which contains $t_i$, $1 \le i \le n$, and does not contain any $t_j$ with $j \neq i$. Assume that a simpleces are added at $t_i$; the case in which they are removed is symmetric. Let $\gamma: ([a, b], a, b) \to (p^{-1}([a, b]), p^{-1}(a), p^{-1}(b))$ be a path, we must show that 
%\begin{enumerate}[label = (\alph*)]
%    \item $\gamma$ is homotopic to a section, $\sigma$, and
%    \item $\sigma$ is unique up to fibrewise homotopy. 
%\end{enumerate}
%To see (a) is true, 
Suppose $[a, b]$ is a subinterval which contains $t_i$, $1 \le i \le n$, and does not contain any $t_j$ with $j \neq i$. Assume that a simpleces are added at $t_i$; the case in which they are removed is symmetric. 

Let us first produce an appropriate vector field on $p^{-1}([a, b])$. Since $p^{-1}([a, b]) = Y_a \times[a, t_i] \cup Y_b \times [t_i, b]$, and $Y_a \subset C_b$ at $t_i$, this is very straightforward: let $\xi$ be the constant vector field on $[a, b]$ which points to the right with magnitude 1. Then $p^*\xi$ is a vector field on $p^{-1}([a, b])$ with the same formal properties as the vector field of Proposition \ref{prp:pseudogradient}. 

The remainder of the proof is entirely identical to the proof of Proposition \ref{prop:tame_are_sectional}.
\end{proof}

Our primary interest, however, is in the uncovered region $X$ of a sensor ball network. To that end, we have
\begin{lem}
Let $\mathcal{D}, C, X, p, SC$ be as above. Then $X$ is locally 0-sectional, subordinate to the same collection of open covers as $SC$. 
\end{lem}
\begin{proof}
$X_{(t_i, t_{i+1})}$ is homotopic to a product, so it suffices to show that $X$ restricted to neighborhoods of the $t_i$, $i = 1, ..., n$, is 0-sectional. Suppose that at $t_i$ we are given an inclusion $SC_{s_{i+1}} \hookrightarrow SC_{s_i}$; the case where we have the reverse inclusion is symmetric. 

Since $\mathcal{D}$ is bounded, $\ol{X_{s_i}}$ and $\ol{X_{s_{i+1}}}$ are compact. Take an open cover by convex sets of $\ol{X_{s_i}}$, and reduce this to a finite subcover $\{U_\alpha\}$. Let $SX_{s_i, U_\alpha}$ denote the \u{C}ech nerve of this complex. Then $X_{(t_i, t_{i+1})} \cong SX_{s_i, U_\alpha} \times (t_i, t_{i+1})$. We can build a \u{C}ech cover of $\ol{X_{s_{i+1}}}$ by starting with (the translates of) the $\{ U_\alpha \}$, and extending it by adjoining to the $\{ U_\alpha \}$ a nice cover of the simplices of $SC_{s_i}$ which are removed at $t_{i+1}$, and then translating along our chosen straightening. 

In fact, we have built a restriction of a stacked simplicial complex $SX|_{(t_i, t_{i+2})} \to (t_i, t_{i+2})$. By Proposition \ref{prop:stacked_are_sectional}, $X$ is therefore locally 0-sectional. 
\end{proof}

\begin{rem}
The deficit of this approach is that we don't get a detailed look at how adding/removing simplices in the covered region affects the same in the uncovered region. Morally, we see that removing a $k$-simplex in $SC$ should be the same as adding a $(d-k)$-simplex in $SX$, as a form of parameterized Alexander Duality. This will be expanded in a future work by the last author. 
\end{rem}

In particular, Theorem \ref{thm:pi0seq} applies, and we have a short exact sequence of sets which computes $\pi_0\Gamma p$. 

\subsection{Sensor ball evasion with homology}\label{subsec:sensors_homology}
By the previous sections, we know that the covered and uncovered regions are both 0-sectional, and so there is a short exact sequence 
\[ * \to R^1 \lim \pi_1(ZX) \to \pi_0 \Gamma p \to \lim \pi_0(ZX) \to *, \]
as in Theorem \ref{thm:pi0seq}, which computes the path components of the section space of the uncovered region. But what if we don't have a good description of $X$? We mentioned in the beginning of this section that we often assume only access to the \u{C}ech nerve of the sensors. Under these more restricted assumptions, we can still obtain a necessary and sufficient condition for $X$ to admit a section. 
\begin{thm}\label{thm:sensor_ball_sections}
Given a sensor ball evasion problem $X \xrightarrow{p} [a, b]$, we can compute the number of solutions as the middle term in the short exact sequence
\[ * \to R^1 \lim \pi_1 ZX \to \pi_0 \Gamma p \to \lim \Hom_{k-alg}(H_{d-1}(ZC), k) \to *. \]
In particular, there exists a section if and only if $lim \Hom_{k-alg}(H_{d-1}(ZC), k)$ is nonempty. 
\end{thm}

There is a natural isomorphism $\pi_0(A) \cong \Hom_{k-alg}(H^0(A), k)$. Thus we have a surjection 
\[ \pi_0 \Gamma p \to \lim \Hom_{k-alg}(H^0(ZX), k). \]
Alexander duality allows us to identify $H^0(X_{s_i})$ with $H_{d-1}(C_{s_i})$. Unfortunately, we cannot directly identify $H^0(X_{s_i}^{s_{i+1}})$ with $H_{d-1}(C_{s_i}^{s_{i+1}})$, which would make the the theorem statement very nice. 

Instead, we should note that for stacked simplicial complexes with a single critical point, there is a deformation retract $X_{s_i}^{s_{i+1}} \to X_{s_i}$ when simplices are removed at $t_i$, and a deformation retract $X_{s_i}^{s_{i+1}} \to X_{s_{i+1}}$ when simplices are added at $t_i$. Thus, in our diagram
\bcd 
... &                         & H^0X_{s_i}^{s_{i+1}} & & ... \\
    & H^0X_{s_i} \ar[ul, leftarrow] \ar[ur, "m", leftarrow] &                   & H^0X_{s_{i+1}} \ar[ul, "n", leftarrow] \ar[ur, leftarrow] & \\
& H_{d-1}C_{s_i} \ar[dl] \ar[dr, "m'"]\ar[u, leftrightarrow, "\sim"] & & H_{d-1}C_{s_{i+1}} \ar[dl, "n'"] \ar[dr] \ar[u, leftrightarrow, "\sim"] & \\
... & & H_{d-1}C_{s_i}^{s_{i+1}} & & ...
\ecd 
either $m$ and $n'$ are both isomorphisms, or $m'$ and $n$ are both isomorphisms. The diagram commutes because Alexander duality is natural with respect to inclusions. 

Happily, we can remove isomorphic terms when taking inverse limits. For instance, if in the above diagram $m$ was an isomorphism, we could replace 
\[ H^0 X_{s_i} \xleftarrow{m} H^0 X_{s_i}^{s_{i+1}} \xrightarrow{n} H^0 X_{s_{i+1}} \]
with
\[ H^0 X_{s_i} \xrightarrow{nm^{-1}} H^0 X_{s_{i+1}}. \]
Associated to this reduced diagram, we have a dual reduced diagram, with arrows
\[ H_{d-1} C_{s_i} \xrightarrow{n'^{-1}m'} H_{d-1} C_{s_{i+1}}. \]
We therefore have an isomorphism of diagrams of $k$-algebras, and so $\lim \Hom_{k-alg}(H_{d-1}(ZC), k)$ computes $\lim \pi_0(ZX)$. This completes the proof of the theorem.

%\todo{I think this part is false, because k-algebras aren't as nice as vector spaces}Taking duals in the category of (finite dimensional) vector spaces over $k$ exchanges finite products with finite coproducts, and preserves exact sequences. Thus the limit over the diagram $H^0(ZX)$ is isomorphic to (the dual of) the colimit over $H_{d-1}(ZC)$. Then
%\begin{align*} 
%\on{colim} \Hom_{k-alg}(H_{d-1}(ZC), k) &\cong \Hom_{k-alg}(\on{colim}H_{d-1}(ZC), k) \\
%&\cong \Hom_{k-alg}(\lim H^0(ZX), k) \\
%&\cong \on{lim}\Hom_{k-alg}(H^0(ZX), k), \end{align*}
%where in the last isomorphism we are using that the diagram and the $k$-algebras are finite dimensional.\todo{I'm actually not sure about commuting this limit when we're in the category of k-algebras. Something like this should work though}

\medskip
\printbibliography

@book{bousfieldkan,
Abstract = {The authors develop a significant and useful general theory of completions and localizations in algebraic topology. We shall summarize their theory chapter by chapter. \par The first half of Part I is concerned with the general theory of $R$-completions ($R$ is a commutative ring with identity). \par In Chapter I the authors associate to a space $X$ a functorial tower of fibrations $\{R_sX\}$ and its inverse limit $R_\infty X$. There is a natural transformation $X\rightarrow R_\infty X$; $R_\infty X$ together with this transformation is called the $R$-completion of $X$. $R$-completions have the following key property: a map $X\rightarrow Y$ induces an isomorphism on reduced $R$ homology $\tilde H_\ast(\ ;R)$ if and only if it induces a homotopy equivalence $R_\infty X\rightarrow R_\infty Y$. The homotopy spectral sequence of the tower $\{R_sX\}$ is the Bousfield-Kan unstable Adams spectral sequence for $R={\bf Z}_p$, while for $R\subset Q$, it consists of the primitive elements in the},
Author = {Bousfield, A. K. and Kan, D. M.},
Publisher = {Springer-Verlag, Berlin-New York},
Series = {Lecture Notes in Mathematics, Vol. 304.},
Title = {Homotopy limits, completions and localizations.},
%URL = {https://search.ebscohost.com/login.aspx?direct=true&db=msn&AN=MR0365573&site=ehost-live&authtype=ip,sso&custid=s4392798},
Year = {1972},
}

@article{EvasionAdamsCarlsson,
  author = {Adams, H. and Carlsson, G.},
  title ={Evasion paths in mobile sensor networks},
  journal = {The International Journal of Robotics Research},
  volume = {34},
  number = {1},
  pages = {90-104},
  year = {2015},
  doi = {10.1177/0278364914548051},
}

@article {MR3763757,
    AUTHOR = {Ghrist, R. and Krishnan, S.},
     TITLE = {Positive {A}lexander duality for pursuit and evasion},
   JOURNAL = {SIAM J. Appl. Algebra Geom.},
  FJOURNAL = {SIAM Journal on Applied Algebra and Geometry},
    VOLUME = {1},
      YEAR = {2017},
    NUMBER = {1},
     PAGES = {308--327},
      ISSN = {2470-6566},
   MRCLASS = {55N30 (55S40 55U30 91A24)},
MRREVIEWER = {Juan Antonio P\'{e}rez},
       DOI = {10.1137/16M1089083},
       URL = {https://doi.org/10.1137/16M1089083},
}

@article {deSilvaGhristEvasionsFence,
    AUTHOR = {de Silva, V. and Ghrist, R.},
     TITLE = {Coordinate-free coverage in sensor networks with controlled boundaries
via homology},
   JOURNAL = {International Journal of Robotics Research},
    VOLUME = {25},
      YEAR = {2006},
     PAGES = {1205--1222}
}

@article {MR2657946,
    AUTHOR = {Carlsson, G. and de Silva, V.},
     TITLE = {Zigzag persistence},
   JOURNAL = {Found. Comput. Math.},
  FJOURNAL = {Foundations of Computational Mathematics. The Journal of the
              Society for the Foundations of Computational Mathematics},
    VOLUME = {10},
      YEAR = {2010},
    NUMBER = {4},
     PAGES = {367--405},
      ISSN = {1615-3375},
   MRCLASS = {68W30 (16G20 55N99)},
       DOI = {10.1007/s10208-010-9066-0},
       URL = {https://doi.org/10.1007/s10208-010-9066-0},
}

@book {MR0190942,
    AUTHOR = {Milnor, J.},
     TITLE = {Lectures on the {$h$}-cobordism theorem},
    SERIES = {Notes by L. Siebenmann and J. Sondow},
 PUBLISHER = {Princeton University Press, Princeton, N.J.},
      YEAR = {1965},
     PAGES = {v+116},
   MRCLASS = {57.10},
MRREVIEWER = {P. E. Conner},
}

@misc{TameFunctionsPaper,
  doi = {10.48550/ARXIV.2006.12023},
  
  url = {https://arxiv.org/abs/2006.12023},
  
  author = {Carlsson, Gunnar and Filippenko, Benjamin},
  
  title = {The space of sections of a smooth function},
  
  publisher = {arXiv},
  
  year = {2020},
  
  copyright = {arXiv.org perpetual, non-exclusive license}
}

@inproceedings{arone2021applying,
  title={Applying calculus of functors to the evasion path problem},
  author={Arone, Gregory Z and Jin, Alvin},
  booktitle={2021 Joint Mathematics Meetings (JMM)},
  year={2021},
  organization={AMS}
}

\end{document}